\documentclass{article}%
\usepackage{amsmath}
\usepackage{amsfonts}
\usepackage{amssymb}
\usepackage{graphicx}%
\providecommand{\U}[1]{\protect\rule{.1in}{.1in}}
\graphicspath{ {\string~/Documents/Images/} }
\providecommand{\U}[1]{\protect\rule{.1in}{.1in}}
\newtheorem{theorem}{Theorem}[section]

\newtheorem{conjecture}{Conjecture}[section]
\newtheorem{corollary}{Corollary}[section]

\newtheorem{definition}{Definition}[section]
\newtheorem{example}{Example}[section]

\newtheorem{lemma}{Lemma}[section]

\newtheorem{proposition}{Proposition}[section]
\newtheorem{remark}{Remark}[section]

\newenvironment{proof}[1][Proof]{\textbf{#1.} }{\ \rule{0.5em}{0.5em}}
\begin{document}

\title{Uniform Edge Betweenness Centrality}
\author{Heather A. Newman\\Department of Mathematics\\Princeton University
\and \ \ \ \ \ \ \ \ \ \ Hector Miranda\\\ \ \ \ \ \ \ \ \ \ \ Mathematical Sciences\\\ \ \ \ \ \ \ \ \ \ \ Rochester Institute of Technology
\and Rigoberto Fl\'{o}rez\\Mathematics and Computer Science\\The Citadel
\and Darren A. Narayan\\Mathematical Sciences\\Rochester Institute of Technology}
\maketitle

\begin{abstract}
The \textit{edge betweenness centrality} of an edge is loosely defined as the
fraction of shortest paths between all pairs of vertices passing through that
edge. In this paper, we investigate graphs where the edge betweenness
centrality of edges is uniform.

It is clear that if a graph $G$ is edge-transitive (its automorphism group
acts transitively on its edges) then $G$ has uniform edge betweenness
centrality. However this sufficient condition is not necessary. Graphs that
are not edge-transitive but have uniform edge betweenness centrality appear to
be very rare. Of the over 11.9 million connected graphs on up to ten vertices,
there are only four graphs that are not edge-transitive but have uniform edge
betweenness centrality. Despite this rarity among small graphs, we present
methods for creating infinite classes of graphs with this unusual combination
of properties.

\end{abstract}

\section{Introduction}

The \textit{betweenness centrality} of a vertex $v$ is the ratio of the number
of shortest paths between two other vertices $u$ and $w$ which contain $v$ to
the total number of shortest paths between $u$ and $w$, summed over all
ordered pairs of vertices $(u,w)$. This idea was introduced by Anthonisse
\cite{Anthonisse} and Freeman \cite{F} in the context of social networks, and
has since appeared frequently in both social network and neuroscience
literature
\cite{Freeman1,Schoch,Brandes,Freeman2,Gago,GagoUB,GagoS,BS,Grassi,White}%
.\newline\qquad We first give some background with some elementary results.

\begin{definition}
The \textbf{edge betweenness centrality} of an edge $e$ in a graph $G$,
denoted $B_{G}^{\prime}(e)$ (or simply $B^{\prime}(e)$ when $G$ is clear),
measures the frequency at which $e$ appears on a shortest path between two
distinct vertices $x$ and $y$. Let $\sigma_{xy}$ be the number of shortest
paths between distinct vertices $x$ and $y$, and let $\sigma_{xy}(e)$ be the
number of shortest paths between $x$ and $y$ that contain $e$. Then
$B_{G}^{\prime}(e)={\displaystyle\sum\limits_{x,y}}\frac{\sigma_{xy}%
(e)}{\sigma_{xy}}$ (for all distinct vertices $x$ and $y$).
\end{definition}

\begin{definition}
We say a graph $G$ has \textbf{uniform edge betweenness centrality}, or is
\textbf{edge-betweenness-uniform}, if $B_{G}^{\prime}(e)$ has the same value
for all edges $e$ in $G$.
\end{definition}

We note that, for undirected graphs, shortest paths from $x$ to $y$ are
regarded as the same as shortest paths from $y$ to $x$, though the associated
contribution to the sum $B_{G}^{\prime}(e)$ is double-counted. In our first
lemma, we state an elementary result on the lower and upper bounds of the
betweenness centrality of an edge.

\begin{lemma}
For a given graph $G$ with $n$ vertices, $2\leq B_{G}^{\prime}(e)\leq
\frac{n^{2}}{4}$ for all vertices $v$ in $G$. Furthermore, these bounds are tight.
\end{lemma}

\begin{proof}
For edge $e$ with end vertices $u$ and $v$, the shortest paths from $u$ and
$v$ and from $v$ to $u$ pass through the edge $e$. Hence $B_{G}^{\prime
}(e)\geq2$. We note that edge betweenness centrality values are larger for
graphs with no cycles, since between any pair of vertices there is a single
path between each pair of vertices. Consider a tree $T$ with $n$ vertices and
a cut-edge $e$. The highest edge betweenness centrality will occur when the
two components of $T-e$ each have $\frac{n}{2}$ vertices. Here $B_{G}^{\prime
}(e)\leq\frac{n^{2}}{4}$.
\end{proof}

Many results on edge betweenness were obtained by Gago \cite{GagoS,GagoUB},
and \cite{Gago}. However, we mention a small oversight in \cite{GagoS}. In the
second part of Lemma 4 in their paper, the following result is stated.

\begin{lemma}
[Comellas and Gago]If $C$ is a cut-set of edges, connecting two sets of
vertices $X$ and $V(G)\backslash X$ and $\left\vert X\right\vert =n_{x}$, then
${\textstyle\sum\limits_{e\in C}} B_{G}^{\prime}(e)=2n_{x}(n-n_{x})$.
\end{lemma}

It is certainly true that ${\textstyle\sum\limits_{e\in C}}B_{G}^{\prime
}(e)\geq2n_{x}(n-n_{x})$ since shortest paths that have one vertex in $X$ and
one vertex in $V(G)\backslash X$ will certainly contain an edge in $C$.
However, equality may not hold, as there may be shortest paths between
vertices in the same part that use an edge in $C$. An example is given below.

\begin{example}
Let $G=P_{3}$ with vertices $u$,$v$, and $w$ and edges $uv$ and $vw$. Let
$C=\{uv,vw\}$ be a cut-set of edges. Then $X=\{v\}$ and $V(G)\backslash
X=\{u,w\}$, and $n_{x}=1$ and $n-n_{x}=2$. The contributions to $B_{G}%
^{\prime}(uv)$ and $B_{G}^{\prime}(vw)$ from paths between the two different
parts is $2(1)(2)=4$. However the contributions from paths between $u$ and $w$
are $2(1)(1)=2$. Hence $B_{G}^{\prime}(uv)+B_{G}^{\prime}(vw)=6$.
\end{example}

\subsection{Elementary Results}

In this subsection, we present a few simple results involving edge betweenness
centrality and uniform edge betweenness centrality. We first give a formula
for the edge betweenness centrality of any edge in a graph of diameter 2.

\begin{proposition}
\noindent Let $G$ be a diameter 2 graph. For any edge $e=uw$ of $G$,%

\[
B^{\prime}(e)=2+2\sum_{v_{i}\in N(w),v_{i}\notin N(u)}\frac{1}{|N(u)\cap
N(v_{i})|}+2\sum_{v_{i}\in N(u),v_{i}\notin N(w)}\frac{1}{|N(w)\cap N(v_{i}%
)|}.
\]

\end{proposition}

\begin{proof}
\noindent Any edge $e=uw$ is on the unique shortest path from $u$ to $w$ and
from $w$ to $u$, contributing 2 to the sum. Since $G$ is of diameter 2, we now
need consider all paths of length 2 containing $e=uw$. Let $v_{i}$ be any
vertex that is distance 2 from $u$ and which is in the neighborhood of $w$,
that is, $v_{i}$ is not a neighbor of $u$. Then there is a unique path from
$u$ to $v_{i}$ containing the edge $e=uw$, namely, the path $u-w-v_{i}$. The
total number of possible paths from $u$ to $v_{i}$ is precisely given by the
number of common neighbors of $u$ and $v_{i}$. The same reasoning applies to
any vertex $v_{i}$ that is distance 2 from $w$, resulting in the above sum.
\end{proof}

\begin{corollary}
\label{11}The uniform edge betweenness centrality of a complete bipartite
graph $K_{m,n}$ is given by:
\[
2+2\left(  \frac{n-1}{m}\right)  +2\left(  \frac{m-1}{n}\right)  .
\]

\end{corollary}

\noindent For even $n$, the uniform edge betweenness centrality of $K_{n}$
minus a perfect matching is given by:
\[
2+\frac{4}{n-2}.
\]
\newline Corollary \ref{11} shows that uniform edge betweenness values are
unbounded, as seen by fixing $n=1$. Corollary 1.2 shows that there are uniform
edge betweenness values infinitely close to the lower bound of 2. We also have
that $B_{G}^{\prime}(e)=2$ if and only if $e$ is a component of $G$.

\section{Edge Transitivity and Uniform Edge Betweenness Centrality}

We recall the definition of vertex (edge) transitivity. A graph is
vertex-transitive (edge-transitive) if its automorphism group acts
transitively on its vertex (edge) set. That is, a graph is vertex-transitive
(edge-transitive) if its vertices (edges) cannot be distinguished from each
other. We will make use of the following alternative definition
\cite{Andersen}.

\begin{theorem}
[Andersen, Ding, Sabidussi, and Vestergaard]A finite simple graph $G$ is
edge-transitive if and only if $G-e_{1}\cong G-e_{2}$ for all pairs of edges
$e_{1}$ and $e_{2}$.
\end{theorem}

\begin{remark}
Clearly if a graph is edge-transitive then it has uniform edge betweenness
centrality. However, the converse is not true.
\end{remark}

We used the databases from Brendan McKay \cite{McKay} and the Wolfram
Mathematica 11 code for betweenness testing and found that of the over 11.9
million graphs on ten vertices or less, there are only four graphs that have
uniform edge betweenness centrality but are not edge-transitive (see Figure 1).%

\begin{center}
\includegraphics[
natheight=0.965100in,
natwidth=3.937500in,
height=0.998in,
width=3.9859in
]%
{OW9UCI04.wmf}%
\\
Figure 1. The four graphs on ten vertices or less that have uniform edge
betweenness centrality but are not edge transitive.
\end{center}

Surprisingly, none of these four graphs are vertex-transitive. We will show
that these properties can arise in a special class of vertex-transitive graphs.

\begin{definition}
A circulant graph $C_{n}(L)$ is a graph on vertices $v_{1},v_{2},...,v_{n}$
where each $v_{i}$ is adjacent to $v_{(i+j)(\operatorname{mod}n)}$ and
$v_{(i-j)(\operatorname{mod}n)}$ for each $j$ in a list $L$. Algebraically,
circulant graphs are Cayley graphs on finite cyclic groups. For a list $L$
containing $m$ items, we refer to $C_{n}(L)$ as an $m$-circulant.
\end{definition}

We note that a circulant graph has rotational symmetry about its vertices and
is therefore vertex-transitive. We will show later that $C_{15}(1,6)$ is not
edge-transitive, but has uniform edge betweenness centrality and this example
can be extended to an infinite class.

We again used the databases from Brendan McKay \cite{McKay} and the Wolfram
Mathematica 11 code for betweenness testing and found that $C_{15}(1,6)$ is
the smallest vertex-transitive graph that has uniform edge betweenness
centrality but is not edge-transitive. For the sake of completeness the
details are given below and in the appendix of this paper.

As noted above none of the four graphs in Figure 1 are vertex-transitive. In
Propositions 2.1-2.5 we next check graphs on 11-15 vertices to identify graphs
that are vertex-transitive, have uniform edge betweenness centrality, but are
not edge-transitive.

\begin{proposition}
There are no edge-betweenness-uniform graphs on 11 vertices that are
vertex-transitive but not edge-transitive.
\end{proposition}

\begin{proof}
There are seven vertex-transitive graphs with 11 vertices.

\begin{enumerate}
\item $C_{11}$ (edge-transitive)

\item $\overline{C_{11}}$ (two different centrality values)

\item $K_{11}$ (edge-transitive)

\item $C_{11}(1,3)$ (two different centrality values)

\item $C_{11}(1,2)$ (two different centrality values)

\item $\overline{C_{11}(1,3)}$ (three different centrality values)

\item $\overline{C_{11}(1,2)}$ (three different centrality values)
\end{enumerate}

These cover all seven cases.
\end{proof}

\begin{proposition}
There are no edge-betweenness-uniform graphs on 12 vertices that are
vertex-transitive but not edge-transitive.
\end{proposition}

\begin{proof}
There are 64 vertex-transitive graphs on 12 vertices, 11 of which are both
edge-transitive and vertex-transitive and the remaining 53 of which are not
edge-betweenness-uniform. (See the appendix for all 64 vertex-transitive
graphs on 12 vertices.)
\end{proof}

\begin{proposition}
There are no edge-betweenness-uniform graphs on 13 vertices that are
vertex-transitive but not edge-transitive.
\end{proposition}

\begin{proof}
There are 13 vertex-transitive graphs on 13 vertices, 4 of which are both
edge-transitive and vertex-transitive and the remaining 9 of which are not
edge-betweenness-uniform. (See the appendix for all 13 vertex-transitive
graphs on 13 vertices.)
\end{proof}

\begin{proposition}
There are no edge-betweenness-uniform graphs on 14 vertices that are
vertex-transitive but not edge-transitive.
\end{proposition}

\begin{proof}
There are 51 vertex-transitive graphs on 14 vertices, 6 of which are both
edge-transitive and vertex-transitive and the remaining 45 of which are not
edge-betweenness-uniform. (See the appendix for all 51 vertex-transitive
graphs on 14 vertices.)
\end{proof}

\begin{proposition}
$C_{15}(1,6)$ is the only graph on 15 vertices which is
edge-betweenness-uniform and vertex-transitive and not edge-transitive. More
specifically, $C_{15}(1,6)$ is the smallest graph satisfying these three conditions.
\end{proposition}

\begin{proof}
There are 44 vertex-transitive graphs on 15 vertices, 10 of which are both
edge-transitive and vertex-transitive and 33 of which are not
edge-betweenness-uniform. The remaining graph is $C_{15}(1,6)$, which is
edge-betweenness-uniform and vertex-transitive, but not edge-transitive. (See
the appendix for all 44 vertex-transitive graphs on 15 vertices.)
\end{proof}

\noindent\newline\indent The only two properties of graphs identified in the
literature thus far that individually imply uniform edge betweenness
centrality are edge transitivity and distance regularity \cite{GagoS2}. Recall
that non-edge-transitive, edge-betweenness-uniform graphs appear to be very
rare. However, we now introduce two infinite classes of graphs that we claim
are edge-betweenness-uniform but neither edge-transitive nor distance-regular.
Specifically, these classes are found among 2-circulants, and thus are
vertex-transitive:
\[
\mbox{Class 1: }C_{18n-3}(1,6n),\mbox{}n\in\mathbb{N}%
\]%
\[
\mbox{Class 2: }C_{18n+3}(1,6n),\mbox{}n\in\mathbb{N}%
\]

\indent We determine that graphs in these classes are neither edge-transitive
nor distance-regular. Not only is this information noteworthy on its own, for
graphs that are edge-betweeenness-uniform but not edge-transitive are rare,
but it also means that we must find a novel heuristic for demonstrating the
uniform edge betweenness centrality of these graphs. Moreover, this is
particularly challenging because counting shortest paths is a nontrivial
problem. Ultimately, our method will \textit{not} require making explicit
shortest paths calculations in demonstrating uniform edge betweenness centrality.

\begin{itemize}
\item We arrange the vertices in graphs in Classes 1 and 2 in a circle, and
label them counterclockwise and consecutively beginning with the label $1$.
The labels are viewed in modulo $18n-3$ for graphs in Class 1 and modulo
$18n+3$ for graphs in Class 2.

\item We refer to edges that connect consecutive vertices as outer chords or
chords of length $1$.

\item We refer to edges that connect vertices that are $6n$ apart as inner
chords or chords of length $6n$.
\end{itemize}

%

\begin{center}
\includegraphics[
natheight=4.347400in,
natwidth=4.111300in,
height=2.8392in,
width=2.6878in
]%
{OW9UCI05.wmf}%
\\
Figure 2. $C_{15}(1,6)$ with outer chords of length 1 (green) and an inner
chords of length 6 (blue)
\end{center}

In reference to 2-circulants, we note for every pair of inner chords, there
exists an automorphism mapping one to another, and for every pair of outer
chords, there exists an automorphism mapping one to another (this is clear by
defining the automorphism as a rotation). Thus, the edges of the graph split
up into at most $2$ orbits. In fact, for all graphs in Classes 1 and 2, the
edges split up into exactly $2$ orbits, since we now show that graphs in these
classes are \textbf{not} edge-transitive (edges are all part of the same
orbit).\medskip\newline Consider a 2-circulant graph $C_{k}(a,b)$. Define
$\lambda_{k}(a,b)$ to be the unique nonnegative integer satisfying
$gcd(k,a)b\equiv\lambda_{k}(a,b)a$ (mod $k$) and let $\Lambda_{k}%
(a,b)=\frac{\lambda_{k}(a,b)}{gcd(k,b)}$.

\begin{lemma}
[Nicoloso and Pietropaoli]\label{NP} Let $C_{k}(a,b)$ and $C_{k}(a^{\prime
},b^{\prime})$ be two (connected) circulants. Without loss of generality,
assume $gcd(k,a)\leq gcd(k,b)$ and $gcd(k,a^{\prime})\leq gcd(k,b^{\prime})$.
Then $C_{k}(a,b)\cong C_{k}(a^{\prime},b^{\prime})$ if and only if one of the
following two conditions holds:

\begin{enumerate}
\item $gcd(k,a)=gcd(k,a^{\prime}) < gcd(k,b) = gcd(k,b^{\prime})$ and
$\Lambda_{k}(a,b) = \Lambda_{k}(a^{\prime},b^{\prime})$

\item $gcd(k,a)=gcd(k,a^{\prime}) = gcd(k,b) = gcd(k,b^{\prime})$ and either
$\Lambda_{k}(a,b)$=$\Lambda_{k}(a^{\prime},b^{\prime})$ or $\Lambda_{k}%
(a,b)$=$\Lambda_{k}(b^{\prime},a^{\prime})$
\end{enumerate}
\end{lemma}

\begin{lemma}
Let $k=18n\pm3$ and let $b=6n$, $n\in\mathbb{N}$. Then the circulant graph
$C_{k}(1,b)$ is not isomorphic to any circulant of the form $C_{k}%
(1,b^{\prime})$ for $b^{\prime}\leq\frac{k}{2}$, $b^{\prime}\neq b$.
\end{lemma}

\begin{proof}%
\[
gcd(k,1)b\equiv\lambda_{k}(1,b)\mbox{ (mod  }k)
\]%
\[
\implies\lambda_{k}(1,b)=b
\]%
\[
\implies\Lambda_{k}(1,b)=\frac{b}{3}%
\]
In order for one of the conditions in Lemma \ref{NP} to be satisfied, we must
have $b=b^{\prime}$.
\end{proof}

\begin{lemma}
[Wilson and Poto\u{c}nik]If $G$ is a tetravalent edge-transitive circulant
graph with $k$ vertices, then either:

\begin{enumerate}
\item $G$ is isomorphic to $C_{k}(1,b)$ for some $b$ such that $b^{2}
\equiv\pm1$ (mod $k$), or

\item $k$ is even, $k=2m$, and $G$ is isomorphic to $C_{2m}(1,m+1)$.
\end{enumerate}
\end{lemma}

\begin{theorem}
Circulant graphs of the form $C_{18n \pm3}(1,6n)$ are \textit{not} edge-transitive.
\end{theorem}

\begin{proof}
By Lemmas 2.2 and 2.3, it suffices to show, letting $k=18n \pm3$ and $b=6n$,
that $b^{2} \not \equiv \pm1$ (mod $k$) and $(k-b)^{2} \not \equiv \pm1$ (mod
$k$). This is easily seen by polynomial long division.
\end{proof}

\noindent\newline We also verify that graphs of the form $C_{18n \pm3}(1,6n)$
are \textit{not} distance-regular by appealing to the following theorem.

\begin{theorem}
[Miklavi\u{c} and Poto\u{c}nik]A Cayley graph of a cyclic group (a circulant)
is distance-regular if and only if it is isomorphic to a cycle, or a complete
graph, or a complete multipartite graph, or a complete bipartite graph on a
twice an odd number of vertices with a matching removed, or the Paley graph
with a prime number of vertices.
\end{theorem}

\begin{corollary}
Circulant graphs of the form $G=C_{18n\pm3}(1,6n)$, $n\in\mathbb{N}$ are not distance-regular.
\end{corollary}

\begin{proof}
It is clear that $G$ is none of the following: a cycle, complete graph, Paley
graph with a prime number of vertices (the order of $G$ is not prime),
complete bipartite graph on twice an odd number of vertices with a matching
removed (the order of $G$ is not even). It is also easily seen that $G$ is not
a complete multipartite graph of the form $K_{s\times t}$. For, $K_{s\times
t}$ has order $st$ and degree sum equal to $s^{2}t(t-1)$. This forces
$st=18n\pm3$ and since $G$ is 4-regular, $s(t-1)=4$, which is impossible.
\end{proof}

\bigskip

We have established that graphs of the form $C_{18n \pm3}(1,6n)$ are neither
edge-transitive nor distance-regular, so we must develop a novel way of
demonstrating their uniform edge betweenness centrality. Importantly, the
method we will introduce does not involve making explicit centrality
calculations by means of counting shortest paths, and will provide insight
into the specific conditions that allow these graphs to be
edge-betweenness-uniform in the absence of the strong condition of edge transitivity.

\indent \newline We first focus on graphs of the form $C_{18n-3}(1,6n)$, and
then easily extend the method to graphs of the form $C_{18n+3}(1,6n)$. We note
two useful facts that we will exploit throughout the proofs:

\begin{enumerate}
\item Circulant graphs are vertex-transitive.

\item To demonstrate the uniform edge betweenness centrality of 2-circulants,
it suffices to show that an inner chord has the same edge betweenness
centrality as an outer chord (since the inner chords occupy the same orbit and
the outer chords occupy the same orbit).
\end{enumerate}

\noindent For $G=C_{15}(1,6)$ $(n=1)$, one can show explicitly that
$B^{\prime}(e)=13$ for all edges $e$ in $G$. We now consider all $n \geq2$.

\begin{lemma}
\label{sixedges}Let $G=C_{18n-3}(1,6n)$ for $n\geq2$ and fix any vertex $s$.
Let $a=(3n-1)(6n)$. Then the only edges that do not lie on a shortest path
originating at $s$ are
$E=\{(s+a,s+a+1),(s-a,s-a-1),(s-a,s+a),(s+a+1,s+a+1-6n),(s-a-1,s-a-1+6n),(s+a+1-6n,s-a-1+6n)\}$%
.
\end{lemma}

\begin{proof}
To find all edges involved in a shortest path originating at vertex $s$, we
use a breadth-first search algorithm. In the algorithm, we mark an edge as
visited if it is part of a shortest path originating at $s$. \noindent\newline

\begin{enumerate}
\item Initialize a boolean-valued array markedEdges() indexed by the edges
$E(G)$ of $G$, so that markedEdges($e$) = false for all $e \in E(G)$.

\item Initialize a boolean-valued array markedVertices() indexed by the
vertices $V(G)$ of $G$, so that markedVertices($v$) = false for all $v \in
V(G)$.

\item Initialize an integer-valued array distTo() indexed by the vertices
$V(G)$ of G, so that distTo$(v)=-1$ for all $v \in V(G)$

\item Initialize a first-in/first-out queue $q$ with operations enqueue and
dequeue, with $q$.enqueue($v$) being the operation that adds $v$ to the queue,
and $q$.dequeue() being the operation that returns the least recently added
item to $q$ and removes it from $q$.
\end{enumerate}

\indent Set markedVertices($s$)=true \newline\indent Set distTo($s$)=0
\newline\indent $q$.enqueue($s$) \newline\indent while $q$ is not empty
\newline\indent \indent  Set $v$ = $q$.dequeue() \newline\indent \indent  for
each $w \in N(v)$ \newline\indent \indent \indent if markedVertices($w$) =
false \newline\indent \indent \indent \indent Set markedVertices($w$) = true
\newline\indent \indent \indent \indent Set markedEdges($(v,w)$) = true
\newline\indent \indent \indent \indent Set distTo($w$)= distTo($v$)+1
\newline\indent \indent \indent \indent $q$.enqueue($w$) \newline%
\indent \indent \indent if distTo($w$) = distTo($v$)+1 \newline%
\indent \indent \indent \indent Set markedEdges($(v,w)$) = true

\noindent\newline Performing the algorithm up to and including marking all
vertices that are distance 3 away from $s$, we find that the marked edges form
the subgraph of $G$ in Figure 3.%

\begin{center}
\includegraphics[
natheight=2.722400in,
natwidth=4.090600in,
height=2.7648in,
width=4.1399in
]%
{OW9UCI06.wmf}%
\\
Figure 3. The subgraph of $G$ consisting of all edges marked by the BFS up to
and including marking all vertices that are distance $3$ away from $s$. The
numbers indicate the order in which vertices were added to the queue $q$, with
$0$ corresponding to the first number added and $20$ corresponding to the most
recent number added.
\end{center}

\noindent\newline Continuing the algorithm until it terminates, we find that
for each $4\leq k\leq3n-1$, the following 12 edges are marked, with all the
same distance from the starting vertex $s$: \noindent\newline\newline%
$(s+(6n)(k-1),s+(6n)(k))$, \newline$(s+(6n)(k),s+(6n)(k)-1)$,\newline%
$(s+(6n)(k-1),s+(6n)(k-1)+1)$, \newline$(s+(6n)(k)+1,s+(6n)(k+1)+1)$,
\newline$(s+(6n)(k+1)+1,s+(6n)(k-2)+1)$, \newline$(s-(6n)(k-1),s-(6n)(k))$,
\newline$(s-(6n)(k),s-(6n)(k)+1)$, \newline$(s-(6n)(k-1),s-(6n)(k-1)-1)$,
\newline$(s-(6n)(k)+1,s-(6n)(k+1)-1)$, \newline$(s-(6n)(k+1)-1,s-(6n)(k-2)-1)$%
, \newline$(s-(6n)(k-2)-1,s-(6n)(k-1)-1)$, and \newline%
$(s+(6n)(k-2)+1,s+6n(k-1)+1)$.

\noindent\newline It follows that the only edges that are not marked by the
algorithm are those in the set $E$ given in the statement of the lemma (see
Figure 4).%

\begin{center}
\includegraphics[
natheight=5.014200in,
natwidth=6.784500in,
height=3.9271in,
width=5.3048in
]%
{OW9UCI07.wmf}%
\\
Figure 4. A generalized depiction of the algorithm performed on $G$, with the
6 edges that remain colored black being the edges that are not marked by the
algorithm (and hence not occupying any shortest path originating at $s$).
\end{center}

\end{proof}

\begin{lemma}
\label{25}Let $G=C_{18n-3}(1,6n)$ for $n\geq2$ and without loss of generality
fix a source vertex $s=2+(6n)(3n-1)$ mod $(18n-3)$ and consider any vertex
$a$. Let $H=G-E$, where $E$ is the set of edges that do not lie on any
shortest path originating at $s$ (Lemma \ref{sixedges}). Then there exists an
an automorphism $\phi$ of $H$ mapping $e_{1}=(a,a+1)$ to $e_{2}=(a,a+6n)$,
given by:
\[
{\small \phi(v)=%
\begin{array}
[c]{cc}%
\Bigg\{ &
\begin{array}
[c]{cc}%
a-(6n)(a-v) & a-a\operatorname{mod}(6n-1)+2\leq v\leq a-a\operatorname{mod}%
(6n-1)+6n\\
\phi(v+1)-1 & v=a-a\operatorname{mod}(6n-1)+1\\
\phi(v-1)+1 & v=a-a\operatorname{mod}(6n-1)+6n+1\\
\phi(v-6n)+1 & a-a\operatorname{mod}(6n-1)+2+6n\leq v\leq
a-a\operatorname{mod}(6n-1)+12n-1\\
\phi(v+6n)-1 & a-a\operatorname{mod}(6n-1)+3-6n\leq v\leq
a-a\operatorname{mod}(6n-1)\\
&
\end{array}
\end{array}
}%
\]
where vertices ($v$ and $\phi(v)$) are taken mod $(18n-3)$.
\end{lemma}

\begin{proof}
By Lemma \ref{sixedges},
$E=\{(12n,12n-1),(1,2),(6n,12n),(1,6n+1),(6n,6n+1),(2,12n-1)\}$.

\noindent\newline To show that $\phi$ is indeed an isomorphism, our strategy
is to show that $N(\phi(v))=\phi(N(v))$.

\noindent\newline\underline{Case 1a:} $a-a\operatorname{mod}%
(6n-1)+2<v<a-a\operatorname{mod}(6n-1)+6n$

\noindent\newline Note that $N(v)=\{v-1,v+1,v+6n,v-6n\}$ \noindent
\newline$\Rightarrow\phi
(N(v))=\{a-(6n)(a-(v-1)),a-(6n)(a-(v+1)),a-(6n)(a-v)+1,a-(6n)(a-v)-1\}$%
.\newline Also $\phi(v)=a-(6n)(a-v)$ \noindent\newline$\Rightarrow
N(\phi(v))=\{a-(6n)(a-v)+1,a-(6n)(a-v)-1,a-6n(a-v)+6n,a-6n(a-v)-6n\}$

$\Rightarrow$ \noindent$N(\phi(v))=\phi(N(v))$.

\noindent\newline\underline{Case 1b:} $v=a-a$ mod $(6n-1)+2$

\noindent\newline Note that $v=2$ or $v=6n+1$ or $v=12n$. By definition of
$E$, \noindent\newline$N(v)=\{v+1,v+6n\}$ \noindent\newline$\Rightarrow
\phi(N(v))=\{\phi(v)+6n,\phi(v)+1\}=N(\phi(v))$.

\noindent\newline\underline{Case 1c:} $v=6n+a-a$ mod $(6n-1)$

\noindent\newline Note that $v=6n$ or $v=12n-1$ or $v=1$. By definition of
$E$, \noindent\newline$N(v)=\{v-1,v-6n\}$ \noindent\newline$\Rightarrow
\phi(N(v))=\{\phi(v)-6n,\phi(v)-1\}=N(\phi(v))$.

\noindent\newline\underline{Case 2:} $v=a-a$ mod $(6n-1)+1$

\noindent\newline Note that $v=1$ or $v=6n$ or $v=12n-1$. By definition of
$E$, \noindent\newline$N(v)=\{v-1,v-6n\}$ \noindent\newline$\Rightarrow
\phi(N(v))=\{\phi(v)-6n,\phi(v)-1\}=N(\phi(v))$.

\noindent\newline\underline{Case 3:} $v=6n+1+a-a$ mod $(6n-1)$

\noindent\newline Note that $v=6n+1$ or $v=12n$ or $v=2$. By definition of
$E$, \noindent\newline$N(v)=\{v+1,v+6n\}$ \noindent\newline$\Rightarrow
\phi(N(v))=\{\phi(v)+6n,\phi(v)+1\}=N(\phi(v))$.

\noindent\newline\underline{Case 4a:} $a-a\operatorname{mod}%
(6n-1)+2+6n<v<a-a\operatorname{mod}(6n-1)+12n-1$

\noindent\newline We have $N(v)=\{v-1,v+1,v+6n,v-6n\}$ and\noindent
\newline$\phi(N(v))=\{\phi(v)-6n,\phi(v)+6n,\phi(v)+1,\phi(v)-1\}=N(\phi(v))$.

\noindent\newline\underline{Case 4b:} $v=a-a\operatorname{mod}(6n-1)+2+6n$

\noindent\newline We have $N(v)=\{v-1,v+1,v+6n,v-6n\}$

$\Rightarrow\phi(N(v))=\{\phi(v)-6n,\phi(v)+1,\phi(v)+6n,\phi(v)-1\}=N(\phi
(v))$.

\noindent\newline\underline{Case 4c:} $v=a-a\operatorname{mod}(6n-1)+12n-1$

\noindent\newline Note that $v=12n-1$ or $v=1$ or $v=6n$. By definition of
$E$, \noindent\newline$N(v)=\{v-1,v-6n\}$ \noindent\newline$\Rightarrow
\phi(N(v))=\{\phi(v)-6n,\phi(v)-1\}=N(\phi(v))$.

\noindent\newline\underline{Case 5a:} $a-a\operatorname{mod}(6n-1)+3-6n<v\leq
a-a\operatorname{mod}(6n-1)$\newline Note that $N(v)=\{v-1,v+1,v-6n,v+6n\}$
\noindent\newline$\Rightarrow\phi(N(v))=\{\phi(v)+6n,\phi(v)-6n,\phi
(v)-1,\phi(v)+1\}=N(\phi(v))$.

\noindent\newline\underline{Case 5b:} $v=a-a\operatorname{mod}(6n-1)+3-6n$

\noindent\newline Note that $v=12n$ or $v=2$ or $v=6n+1$. By definition of
$E$, \noindent\newline$N(v)=\{v+6n,v+1\}$ \noindent\newline$\Rightarrow
\phi(N(v))=\{\phi(v)+1,\phi(v)+6n\}=N(\phi(v))$.

\noindent\newline We have shown that $\phi$ is an isomorphism. Since
$\phi(a)=a$ and $\phi(a+1)=a-(6n)(a-(a+1))=a +6n$, $\phi$ maps the edge
$e_{1}=(a,a+1)$ to the edge $e_{2}=(a,a+6n)$.
\end{proof}

\begin{corollary}
\label{22}Let $G=C_{18n-3}(1,6n)$ and fix a vertex $s$. Let $E_{s\prime}$ be
the set of edges that do not lie on a shortest path originating at $s^{\prime
}$ (Lemma \ref{sixedges}). Then for every vertex $s^{\prime}$, there exists an
automorphism of $G-E_{s\prime}$ mapping $(s,s+1)$ to $(s,s+6n)$.
\end{corollary}

\begin{proof}
The corollary follows from Lemma 2.5 along with the fact that $G$ is vertex-transitive.
\end{proof}

\begin{theorem}
Circulant graphs of the form $C_{18n-3}(1,6n)$ are edge-betweenness-uniform.
\end{theorem}

\begin{proof}
Recall that since $G$ is not edge-transitive, its edges fall into two
different orbits: chords of length $1$ and chords of length $6n$. Thus, all
chords of length $1$ have the same edge betweenness centrality value, and all
chords of length $6n$ have the same edge betweenness centrality value. Thus,
to show that $G$ has uniform edge betweenness centrality, it suffices to show
that a chord of length $1$ has the same edge betweenness centrality as a chord
of length of $6n$. Without loss of generality, we can choose any two chords
with these lengths.

\noindent\newline Fix any vertex $s$. We would like to show that the edges
$(s,s+1)$ (a chord of length $1$) and $(s,s+6n)$ (a chord of length $6n$) have
the same edge betweenness centrality. By Lemma \ref{sixedges}, we know that
the only edges that do not lie on a shortest path originating at $s$ are those
in the set $E_{s}%
=\{(s+a,s+a+1),(s-a,s-a-1),(s-a,s+a),(s+a+1,s+a+1-6n),(s-a-1,s-a-1+6n),(s+a+1-6n,s-a-1+6n)\}$%
, where $a=(3n-1)(6n)$.

\noindent\newline Recall that the edge betweenness centrality of an edge $e$
is defined by the formula:
\[
B^{\prime}(e) = \sum_{x,y}\frac{\sigma_{xy}(e)}{\sigma_{xy}}
\]
for distinct vertices $x, y$.

\noindent\newline Without loss of generality, fix the pair of edges
$e_{1}=(s,s+1)$ and $e_{2}=(s,s+6n)$. Since no edges in the set $E$ are on
shortest paths originating at $s$, we can ignore these edges in calculating
the contributions to $B^{\prime}(e_{1})$ and $B^{\prime}(e_{2})$ from shortest
paths originating at $s$. Since by Corollary 2.2 there exists an automorphism
of $G-E_{s}$ mapping $e_{1}$ to $e_{2}$, we have that
\[
\sum_{y\in V(G)}\frac{\sigma_{sy}(e_{1})}{\sigma_{sy}}=\sum_{y\in V(G)}%
\frac{\sigma_{sy}(e_{2})}{\sigma_{sy}}\text{.}%
\]
In other words, the contributions to $B^{\prime}(e_{1})$ and $B^{\prime}%
(e_{2})$ from all shortest paths originating at the vertex $s$ are equal.

\noindent\newline Moreover, by Corollary 2.2, for every vertex $s^{\prime}$,
we can map $(s,s+1)$ to $(s,s+6n)$ through an automorphism of $G-E_{s^{\prime
}}$, so that the contributions to $B^{\prime}(e_{1})$ and $B^{\prime}(e_{2})$
from all shortest paths originating at the vertex $s^{\prime}$ are equal. Therefore:%

\[
\sum_{s^{\prime}\in V(G)}\sum_{y\in V(G)}\frac{\sigma_{s^{\prime}y}(e_{1}%
)}{\sigma_{s^{\prime}y}}=\sum_{s^{\prime}\in V(G)}\sum_{y\in V(G)}\frac
{\sigma_{s^{\prime}y}(e_{2})}{\sigma_{s^{\prime}y}}%
\]%
\[
\Rightarrow B^{\prime}(e_{1})=B^{\prime}(e_{2})\text{.}%
\]
\noindent\newline This completes the proof.
\end{proof}

\bigskip

The process of demonstrating the uniform edge betweenness centrality of graphs
of the form $C_{18n+3}(1,6n)$ is essentially the same, once we have identified
the analogues of Lemma \ref{sixedges} and Lemma \ref{25} and Corollary
\ref{22}, which we present below.

\noindent\newline For $G=C_{21}(1,6)$ $(n=1)$, one can show explicitly that
$B^{\prime}(e)=22$ for all edges $e$ in $G$. We now consider all $n \geq2$.

\begin{lemma}
\label{26}Let $G=C_{18n+3}(1,6n)$ for $n\geq2$ and fix any vertex $s$. Let
$a=(3n)(6n)$. Then the only edges that do not lie on a shortest path
originating at $s$ are
$E=\{(s+a,s+a-1),(s-a,s+a),(s-a,s-a+1),(s+a-1,s+a-1-6n),(s-a+1,s-a+1+6n),(s+a-1-6n,s-a+1+6n)\}$%
.
\end{lemma}

\begin{lemma}
Let $G=C_{18n+3}(1,6n)$ for $n\geq2$ and without loss of generality fix a
source vertex $s=2+(6n)(3n)$ mod $(18n+3)$ and consider any vertex $a$. Let
$H=G-E$, where $E$ is the set of edges that do not lie on any shortest path
originating at $s$ (Lemma \ref{26}). Then there exists an automorphism $\phi$
of $H$ mapping $e_{1}=(a,a-1)$ to $e_{2}=(a,a+6n)$, given by:
\[
{\small \phi(v)=%
\begin{array}
[c]{cc}%
\Bigg\{ &
\begin{array}
[c]{cc}%
a+(6n)(a-v) & a-a\operatorname{mod}(6n+1)+3\leq v\leq a-a\operatorname{mod}%
(6n+1)+6n+1\\
\phi(v+1)+12n+1 & v=a-a\operatorname{mod}(6n+1)+2\\
\phi(v+1)+6n & v=a-a\operatorname{mod}(6n+1)+1\\
\phi(v+(6n+1))+(6n+1) & a-a\operatorname{mod}(6n+1)-6n\leq v\leq
a-a\operatorname{mod}(6n+1)\\
\phi(v-(6n+1))-(6n+1) & a-a\operatorname{mod}(6n+1)+6n+2\leq v\leq
a-a\operatorname{mod}(6n+1)+12n+2\\
&
\end{array}
\end{array}
}%
\]
where vertices ($v$ and $\phi(v)$) are taken mod $(18n+3)$.
\end{lemma}

\begin{proof}
By Lemma \ref{26},
$E=\{(2,3),(12n+5,12n+4),(2,12n+5),(6n+4,6n+3),(12n+4,6n+4),(3,6n+3)\}$.

\noindent\newline To show that $\phi$ is indeed an isomorphism, our strategy
is to show that $N(\phi(v))=\phi(N(v))$.

\noindent\newline\underline{Case 1a:} $a-a\operatorname{mod}%
(6n+1)+3<v<a-a\operatorname{mod}(6n+1)+6n+1$

\noindent\newline We note that $N(v)=\{v-1,v+1,v-6n,v+6n\}$ \noindent
\newline$\Rightarrow\phi
(N(v))=\{a+(6n)(a-(v-1)),a+(6n)(a-(v+1)),a+(6n)(a-v)+1,a+(6n)(a-v)-1\}$.

\noindent\newline Then $\phi(v)=a+(6n)(a-v)$ \noindent\newline$\Rightarrow
N(\phi(v))=\{a+(6n)(a-v)+1,a+(6n)(a-v)-1,a+6n(a-v)+6n,a+6n(a-v)-6n\}$.

\noindent\newline Hence $N(\phi(v))=\phi(N(v))$.

\noindent\newline\underline{Case 1b:} $v=a-a$ mod $(6n+1)+3$

\noindent\newline Note that $v=3$ or $v=6n+4$ or $v=12n+5$. By definition of
$E$, \noindent\newline$N(v)=\{v+1,v-6n\}$ \noindent\newline$\Rightarrow
\phi(N(v))=\{\phi(v)-6n,\phi(v)+1\}=N(\phi(v))$.

\noindent\newline\underline{Case 1c:} $v=a-a$ mod $(6n+1) + 6n +1$

\noindent\newline Note that $v=6n+1$ or $v=12n+2$ or $v=18n+3$. By definition
of $E$, \noindent\newline$N(v)=\{v-1,v+1,v-6n,v+6n\}$ \noindent\newline%
$\Rightarrow\phi(N(v))=\{\phi(v)+6n,\phi(v)-6n,\phi(v)+1,\phi(v)-1\}=N(\phi
(v))$.

\noindent\newline\underline{Case 2:} $v=a-a$ mod $(6n+1)+2$

\noindent\newline Note that $v=2$ or $v=6n+3$ or $v=12n+4$. By definition of
$E$, \noindent\newline$N(v)=\{v-1,v+6n\}$ \noindent\newline$\Rightarrow
\phi(N(v))=\{\phi(v)+6n,\phi(v)-1\}=N(\phi(v))$.

\noindent\newline\underline{Case 3:} $v=a-a$ mod $(6n+1)+1$

\noindent\newline Note that $N(v)=\{v-1,v+1,v-6n,v+6n\}$ \noindent
\newline$\Rightarrow\phi(N(v))=\{\phi(v)+6n,\phi(v)-6n,\phi(v)+1,\phi
(v)-1\}=N(\phi(v))$.

\noindent\newline\underline{Case 4a:} $a-a\operatorname{mod}%
(6n+1)+3-(6n+1)<v<a-a\operatorname{mod}(6n+1)+(6n+1)-(6n+1)$

\noindent\newline Note that $N(v)=\{v-1,v+1,v+6n,v-6n\}$ \noindent
\newline$\Rightarrow\phi(N(v))=\{\phi(v)+6n,\phi(v)-6n,\phi(v)+1,\phi
(v)-1\}=N(\phi(v))$.

\noindent\newline\underline{Case 4b:} $v=a-a\operatorname{mod}(6n+1)+3+(6n+1)$

\noindent\newline Note that $v=3$ or $v=6n+4$ or $v=12n+5$. By definition of
$E$, \noindent\newline$N(v)=\{v+1,v-6n\}$ \noindent\newline$\Rightarrow
\phi(N(v))=\{\phi(v)-6n,\phi(v)+1\}=N(\phi(v))$.

\noindent\newline\underline{Case 4c:} $v=a-a\operatorname{mod}(6n+1)+2+(6n+1)$

\noindent\newline Note that $v=2$ or $v=6n+3$ or $v=12n+4$. By definition of
$E$, \noindent\newline$N(v)=\{v-1,v+6n\}$ \noindent\newline$\Rightarrow
\phi(N(v))=\{\phi(v)+6n,\phi(v)-1\}=N(\phi(v))$.

\noindent\newline\underline{Case 4d:} $v=a-a\operatorname{mod}(6n+1)+1+(6n+1)$

\noindent\newline Note that $N(v)=\{v-1,v+1,v-6n,v+6n\}$ \noindent
\newline$\Rightarrow\phi(N(v))=\{\phi(v)+6n,\phi(v)-6n,\phi(v)+1,\phi
(v)-1\}=N(\phi(v))$.

\noindent\newline\underline{Case 4e:} $v=a-a\operatorname{mod}(6n+1)$

\noindent\newline Note that $N(v)=\{v-1,v+1,v-6n,v+6n\}$ \noindent
\newline$\Rightarrow\phi(N(v))=\{\phi(v)+6n,\phi(v)-6n,\phi(v)+1,\phi
(v)-1\}=N(\phi(v))$.

\noindent\newline\underline{Case 5:} $a-a\operatorname{mod}(6n+1)+6n+2\leq
v\leq a-a\operatorname{mod}(6n+1)+12n+2$ \noindent\newline The proof reduces
to the proof of Case 4.

\noindent\newline We have shown that $\phi$ is an isomorphism. Since
$\phi(a)=a$ and $\phi(a-1)=a+(6n)(a-(a-1))=a+6n$, $\phi$ maps the edge
$e_{1}=(a,a-1)$ to the edge $e_{2}=(a,a+6n)$.
\end{proof}

\bigskip

\begin{corollary}
Let $G=C_{18n+3}(1,6n)$ and fix a vertex $s$. Let $E_{s^{\prime}}$ be the set
of edges that do not lie on a shortest path originating at $s^{\prime}$ (Lemma
\ref{26}). Then for every vertex $s^{\prime}$, there exists an automorphism of
$G-E_{s^{\prime}}$ mapping $(s,s-1)$ to $(s,s+6n)$.
\end{corollary}

\begin{theorem}
Circulant graphs of the form $C_{18n+3}(1,6n)$ are edge-betweenness-uniform.
\end{theorem}

\noindent\newline We have identified seven other infinite classes which we
believe have the same unusual combination of properties. We pose these as open problems.

\begin{conjecture}
The following classes of circulant graphs have uniform edge betweenness
centrality but are not edge-transitive: \newline\newline Class 3:
$C_{20+8(n-1)}(1,2n+2,2n+4)$, $n=1,2,3,\dots\newline$\newline Class 4:
$C_{32+8(n-1)}(1,2n+5,2n+7)$, $n$ $=1,2,3,\dots\newline$\newline Class 5:
$C_{20+16(n-1)}(1,4n,8n+1)$, $n$ $=1,2,3,\dots\newline$\newline Class 6:
$C_{28+16(n-1)}(1,4n+4,8n+5)$, $n$ $=1,2,3,\dots\newline$\newline Class 7:
$C_{32+8(n-1)}(1,2n+5,4n+11)$, $n$ $=1,2,3,\dots\newline$\newline Class 8:
$C_{32+8(n-1)}(1,2n+7,4n+11)$, $n$ $=1,2,3,\dots\newline$\newline Class 9:
$C_{49+14(n-1)}(1,2n+6,4n+9)$, $n$ $=1,2,3,\dots$
\end{conjecture}

{\large Acknowledgements\bigskip}

The authors are very grateful to Stanis\l {}aw Radziszowski for useful
discussion and for processing graph data on edge transitive graphs up to 20
vertices. Research was supported by National Science Foundation Research
Experiences for Undergraduates Site Award \#1659075.

\end{document}